\documentclass{amsart}
\usepackage[utf8]{inputenc}
\usepackage[all]{xy}
\usepackage{upgreek,amsthm,bbm,bm}
\usepackage{color,amssymb, comment, mathrsfs,tikz,amsmath,amsfonts,stmaryrd,tikz-cd,hyperref,tikz-cd}
\usepackage{colonequals}
\usetikzlibrary{matrix}
\usetikzlibrary{arrows}

\makeatletter
\@namedef{subjclassname@2020}{\textup{2020} Mathematics Subject Classification}
\makeatother

\hypersetup{
    colorlinks=true,
    linkcolor=blue,
    filecolor=magenta,      
    urlcolor=cyan,
} 

\usepackage[capitalise]{cleveref}
\crefformat{equation}{(#2#1#3)}
\crefrangeformat{equation}{(#3#1#4--#5#2#6)}
\crefformat{enumi}{(#2#1#3)}
\crefrangeformat{enumi}{(#3#1#4--#5#2#6)}

\newcommand{\hopf}{\mathcal{H}}

\newtheorem{theorem}{Theorem}[section]
\newtheorem{question}[theorem]{Question}
\newtheorem{lemma}[theorem]{Lemma}
\newtheorem{proposition}[theorem]{Proposition}
\newtheorem{corollary}[theorem]{Corollary}

\newcounter{intro}
\newtheorem{introthm}[intro]{Theorem}

\theoremstyle{definition}

\newtheorem{example}[theorem]{Example}

\newtheorem{notation}[theorem]{Notation}
\newtheorem{chunk}[theorem]{}
\newtheorem{remark}[theorem]{Remark}
\newtheorem*{ack}{Acknowledgements}

\newcommand{\AQC}[3]{\operatorname{D}(#1/#2;#3)}

\newcommand{\F}[1]{\mathrm{F}^{#1}}

\newcommand{\del}{\partial}

\newcommand{\xla}[1]{\xleftarrow{#1}}
\newcommand{\xra}[1]{\xrightarrow{#1}}

\newcommand{\shift}{\mathsf{\Sigma}}

\newcommand{\Ext}{\operatorname{Ext}}
\newcommand{\Tor}{\operatorname{Tor}}
\newcommand{\Der}{\operatorname{Der}}
\newcommand{\ind}{\operatorname{ind}}

\newcommand{\hh}{\mathrm{H}}

\newcommand{\Hom}{\operatorname{Hom}}

\newcommand{\cls}[1]{\operatorname{cls}(#1)}

\newcommand{\m}{\mathfrak{m}}
\newcommand{\vp}{\varphi}
\newcommand{\ve}{\varepsilon}
\newcommand{\lotimes}{\otimes^\mathbf{L}}

\newcommand{\g}{\mathfrak{g}}
\newcommand{\h}{\mathfrak{h}}
\newcommand{\f}{\mathfrak{f}}
\newcommand{\im}{\mathrm{Im}}

\newcommand{\Ker}{\mathrm{Ker}}

\title[The homotopy Lie algebra of a tensor product]{The homotopy Lie algebra of a Tor-independent tensor product}

\author[L.~Ferraro]{Luigi Ferraro}
\address{Department of Mathematics,
Texas Tech University, Lubbock, TX 79409, U.S.A.}
\email{lferraro@ttu.edu}

\author[M.~Gheibi]{Mohsen Gheibi}
\address{Department of Mathematics,
Florida A\&M University,
Tallahassee, FL 32307 }
\email{mohsen.gheibi@famu.edu}

\author[D.~A.~Jorgensen]{David A. Jorgensen}
\address{Department of Mathematics, University of Texas,  Arlington, TX 76019, USA}
\email{djorgens@uta.edu}

\author[N.~Packauskas]{Nicholas Packauskas}
\address{Department of Mathematics,
SUNY Cortland, Cortland, NY 13045, U.S.A.}
\email{nicholas.packauskas@cortland.edu}

\author[J.~Pollitz]{Josh Pollitz}
\address{Department of Mathematics,
University of Utah, Salt Lake City, UT 84112, U.S.A.}
\email{pollitz@math.utah.edu}

\date{\today}

\begin{document}

\begin{abstract}
In this article we investigate a pair of surjective local ring maps  $S_1\leftarrow R\to S_2$ and their relation to the canonical projection $R\to S_1\otimes_R S_2$, where $S_1,S_2$ are Tor-independent over $R$.
Our main result asserts a  structural  connection between the homotopy Lie algebra of $S\colonequals S_1\otimes_R S_2$, denoted $\pi(S)$, in terms of those of $R,S_1$ and $S_2$. Namely,  $\pi(S)$ is the pullback of (adjusted) Lie algebras along the maps $\pi(S_i)\to \pi(R)$ in various cases, including when the maps above have residual characteristic zero. Consequences to the main  theorem include structural results on  Andr\'{e}--Quillen cohomology,  stable cohomology, and Tor algebras, as well as an equality relating the Poincar\'{e} series of the common residue field of
$R,S_1,S_2$ and $S$.
\end{abstract}

\maketitle
\section*{Introduction}
Given a pair of Tor-independent modules over a commutative ring, there are relationships between homological properties of the modules and their tensor product.   Building upon work in  \cite{tensor,small,JM,Keller,Vasconcelos}, we focus on surjective local ring maps whose targets are Tor-independent over their common source; the minimal intersection rings introduced in \cite{JM} served as the starting place for the present investigation. The main result of this article establishes an especially strong connection between homotopical aspects of the ring maps and their tensor product.

Let $R$ be a commutative  noetherian local ring with residue field $k$. Let $S_1, S_2$ be quotients of $R$ which are Tor-independent, that is,  $\Tor^R_i(S_1, S_2) = 0$ for $i > 0$.  With inspiration from \cite{tensor}, setting $S = S_1 \otimes_R S_2$, this paper examines relationships between  $R,$  $S_1,$  $S_2$ and $S$ through the lens of their homotopy Lie algebras. We write $\pi(R)$ for the homotopy Lie algebra of $R$, which is a graded Lie $k$-algebra naturally associated  to the ring $R$.  We recall its construction in  \cref{HomotopyLieAlgebra}. These homotopy Lie algebras have been adopted from rational homotopy theory to local algebra by Avramov~\cite{Asterisque}, and their importance is, in part, due to the ring-theoretic properties they encode; see, for example, \cite{Golod,AH:1986,AH:1987,Briggs,Briggs:2021,FiveAuthorPaper} as well as \cref{p:parallelmaps}.

Our main result is \cref{introthm1}, below; before stating it we set notation and terminology. We write $\g\times_{\h}\g'$ for the pullback along Lie algebra maps $\g\to \h\leftarrow \g'.$ Our theorem applies to surjective local maps of residual characteristic zero, as well as two large classes of ring maps having slightly technical definitions: namely maps that are \emph{almost small} or have \emph{finite weak category} recalled in \cref{c:almostsmall} and \cref{d:WeakCat}, respectively. The former generalizes the small homomorphisms of Avramov~\cite{small} and were introduced in \cite{AlgebraRetracts}, while the latter have appeared in \cite{Flat,AlgebraRetracts,Briggs,rational,GL}; both contain the class of maps whose source is regular.

\begin{introthm}\label{introthm1}
Suppose $S_1\xla{\vp_1} R\xra{\vp_2}S_2$ is a Tor-independent pair of surjective local ring maps  with residue field $k$, and set $S=S_1\otimes_R S_2$. Consider the following  conditions:
\begin{enumerate}
  \item $k$ has characteristic zero;
    \item at least one $\vp_i$ is almost small;
    \item each $\vp_i$ has finite weak category.
\end{enumerate} If any of the conditions above hold, then there is a naturally induced isomorphism of graded Lie algebras \[\pi(S)\cong \pi(S_1)\times_{\pi(R)}\pi(S_2)\,.\]
\end{introthm}

Under more restrictive hypotheses, the conclusion of \cref{introthm1} can be asserted from two results of Avramov: first, when $R$ is regular and a  further assumption is imposed on the kernels of the $\vp_i$  in \cite{tensor}, and second, when one of the $\vp_i$ is assumed to be small in \cite{small}. Both arguments require a careful analysis of certain spectral sequences while further relying on celebrated theorems of Andr\'{e}, Milnor--Moore, and Sj\"{o}din in \cite{Andre,MM,Sjodin}, respectively. Besides the greater generality of  \cref{introthm1}, its proof is direct, given by an examination of DG $R$-algebra resolutions of $S_1$, $S_2$ and $S$. We direct the reader to \cref{thm:SurjPiIsoAlmostSmall} and \cref{thm:surjPiIsochar0}, which in conjunction establish \cref{introthm1}; see also \cref{remark:cmonman}.

The paper is organized as follows.  \cref{sec:DG} recalls preliminaries regarding DG algebras, especially minimal models and homotopy Lie algebras.  In  \cref{sec:torind} we provide general consequences of Tor-independence  and shared properties of the projections $R \to S_i$ and $S_i \to S$ by examining the fibers on the induced maps on minimal models; see \cref{p:ci,p:parallelmaps,p:Koszul}.

 \cref{sec:AS,sec:char0} are the heart of the present article, as they  contain the proofs of the main result above. \cref{sec:applications} is the final section, which provides several applications of  \cref{introthm1} including results regarding the structure of stable cohomology modules, Tor algebras and  Andr\'{e}--Quillen cohomology modules   for Tor-independent maps; see \cref{rem:stable,cor:TorAlg,cor:AQ}, respectively.
An application in a different direction is a numerical relationship on the Poincar\'e series of $R$, $S_1,$ $S_2$ and $S$ which is presented in  \cref{cor:PSeriesAS}, and stated below.

\begin{introthm}
\label{introthm2}
Suppose $S_1\xla{\vp_1}R\xra{\vp_2} S_2$ is a pair of surjective local ring maps with residue field $k$ and  Tor-independent targets, and let $S=S_1\otimes_R S_2$. If at least one $\vp_i$ is almost small, then there is an equality of formal power series
\[\mathrm{P}^S_k(t)=\frac{\mathrm{P}^{S_1}_k(t)\cdot \mathrm{P}^{S_2}_k(t)}{\mathrm{P}^R_k(t)}\,.\]
\end{introthm}

This theorem is of particular interest as it generalizes  results from \cite{tensor,small}, while its conclusion is only established under one of the three specified conditions in \cref{introthm1}; cf.\@ \cref{question}.


\begin{ack}
We thank Benjamin Briggs, Michael DeBellevue and Srikanth Iyengar for helpful conversations regarding this work. We thank an anonymous referee for pointing out a mistake in a previous version of the manuscript; our corrections have led to simplifications in several arguments, specifically in \cref{sec:char0}. Finally, the fifth author was partially supported by NSF grant DMS 2002173. 
\end{ack}

\section{DG algebras and homotopy Lie algebras}
\label{sec:DG}

Throughout this section, let $(R,\m,k)$ be a local ring and fix  a surjective local homomorphism $\vp\colon R\to S$. We recall the necessary background regarding certain semifree DG algebra resolutions as well as the homotopy Lie algebra;  suitable references include \cite{Asterisque,IFR,Briggs}.

By convention all DG algebras will be  nonnegatively graded, strictly graded commutative and local. That last condition means that for a DG algebra $A$, the base ring $A_0$ is a commutative noetherian local ring, which we assume also has residue field $k$, and each $\hh_i(A)$ is a finitely generated $\hh_0(A)$-module; throughout $A$ will denote such a DG algebra. We will use $\del$ to denote the differential of a complex.

\begin{chunk} A \emph{semifree extension} of $R$ is a DG $R$-algebra $R[X]$ where
$X=X_{\geqslant 1}$ is a graded set of variables consisting of exterior variables in each odd degree and polynomial variables in each even degree. 
By a slight abuse of notation we write $k[X]$ for the DG $k$-algebra $k\otimes_R R[X].$
We say $R[X]$ is a \emph{minimal semifree extension} of $R$ provided there is the containment $\del(k[X])\subseteq (X)^2.$
\end{chunk}

\begin{chunk}\label{MinModel}
A \emph{minimal model for $\vp$} is a semifree extension $R[X]$ fitting into a factorization of $\vp$ as $R\to R[X]\xra{\tilde{\vp}} S$ satisfying the following:
\begin{enumerate}
    \item $R[X]$ is a minimal semifree extension of $R$;
    \item $\tilde{\vp}\colon R[X]\to S$ is a surjective quasi-isomorphism.
\end{enumerate}  
Moreover, minimal models always exist and are unique up to homotopy equivalence of DG $R$-algebras;  see \cite[Section~7.2]{IFR} as well as \cite[Section~2.7]{Briggs}.
\end{chunk}

\begin{chunk}\label{c:acyclicclsoure}
 A \emph{semifree $\Gamma$-extension of }$A$ is a DG
algebra $A\langle Y\rangle$ where $Y=Y_{\geqslant 1}$ is a graded set of variables consisting of exterior variables in each odd degree and divided powers variables in each positive even degree; see, for example, \cite[Section~6.1]{IFR}. 
 
Assume there is a surjection $A\to S$ and let $J=\Ker(A\to S).$
 Following \cite[Construction~6.3.1]{IFR}, an  \emph{acyclic closure of $S$ over }$A$ is a semifree  $\Gamma$-extension $A\langle Y \rangle$ of $A$ with $\hh(A\langle Y\rangle)=S$ such that
\begin{enumerate}
\item $\del(Y_1)$ minimally generates $J_0\mod \del(A_1)$;
\item $\{\cls{\del(y)} \mid y\in Y_{n+1}\}$ minimally generates $\hh_n(A\langle Y_{\leqslant n} \rangle)$
for $n\geqslant 1$.
\end{enumerate}
Existence of an acyclic closure is the content of \cite[Proposition~1.9.3]{GL}, while uniqueness up to DG $\Gamma$-algebra isomorphism is established in \cite[Theorem~1.9.5]{GL}.
\end{chunk}

\begin{chunk}\label{Derivations}  Let $A\langle Y \rangle$ be a semifree  $\Gamma$-extension and  $M$ be a  DG $A\langle Y\rangle$-module. An $A$-linear map $d\colon A\langle Y\rangle\to M$ is called an
\emph{$\Gamma$-derivation} if it satisfies
\begin{enumerate}
\item $d(bb')=d(b)b'+(-1)^{|b||d|}bd(b')$ for all $b,b'\in A\langle Y\rangle$;
\item $d(y^{(i)})=d(y)y^{(i-1)}$ for all $y\in Y_{2i}$ and all $i\geqslant 1$.
\end{enumerate}
We denote the collection of $A$-linear $\Gamma$-derivations from $A\langle Y\rangle$ to $M$ by 
$\Der^\gamma_A(A\langle Y\rangle,M)$.
\end{chunk}

\begin{chunk}\label{HomotopyLieAlgebra}
A graded Lie algebra over $k$ is a graded $k$-vector space  equipped with a $k$-bilinear pairing, called its Lie bracket, and a squaring operation  satisfying a list of axioms specified in \cite[Remark~10.1.2]{IFR}.

The \emph{homotopy Lie algebra of $A$}, denoted by $\pi(A)$, is defined to be
\[
\pi(A)=\hh\left(\Der^\gamma_A(A\langle Y\rangle,A\langle Y\rangle)\right)\,,
\] 
where $A\langle Y\rangle$ is an acyclic closure of $k$ over $A$. The Lie bracket is the graded commutator and the square of an element is the composition of that element with itself. As any two acyclic closures are isomorphic as semifree $\Gamma$-extensions, see \cref{c:acyclicclsoure}, $\pi(A)$ is  independent of choice of acyclic closure for $k$ over $A$. 
Furthermore, $\pi(A)$ is naturally identified as a $k$-subspace of $\Ext_A(k,k)$ with the latter being the universal enveloping algebra of $\pi(A)$; see \cite[Theorem~10.2.1]{IFR}. Finally, given a quasi-isomorphism of local DG algebras $A\xra{\simeq} B$ it follows from \cite[Lemma~7.2.10]{IFR} that $\pi(B)\xra{\cong}\pi(A)$ as graded Lie algebras over $k$.

When $A$ is the local ring $(R,\m,k)$ and $R\langle Y\rangle$ is an acyclic closure of $k$ over $R$, it is  well-known that the quasi-isomorphism  
\[
R\langle Y\rangle \to \widehat{R}\otimes_R R\langle Y\rangle
\] induces an isomorphism of graded Lie algebras $\pi(\widehat{R})\xra{\cong} \pi(R)$, where $\widehat{R}$ denotes the $\m$-adic completion of $R$. 
Moreover, the surjective map $\vp\colon R\to S$ induces a map of graded Lie algebras over $k$:
\[
\pi(\vp)\colon \pi(S)\to \pi(R)\,.
\] See, for example, \cite[Remark~10.2.4]{IFR} or \cite{GL} for more details on the construction of this map. 
\end{chunk}

\begin{chunk} 
Let $B=R[X]$ be a semifree extension of $R$. Its indecomposable complex, denoted $\ind_R R[X]$, is the complex of free $R$-modules 
\[
\dfrac{B}{B_0+(X)^2}\cong \ldots \to RX_n\to RX_{n-1}\to \ldots \to RX_2\to RX_1\to 0\,.
\]
Minimality of $B$ is detected using $\ind_k k[X]$. Namely, $B$ is a minimal semifree extension of $R$ if and only if $\ind_k k[X]$ has trivial differential. 

On the other hand, given a semifree $\Gamma$-extension  $A\langle Y \rangle$ of $A$ with $S=\hh_0(A\langle Y\rangle)$, its complex of $\Gamma$-indecomposables (with respect to $A$), denoted  
$\ind_A^\gamma A\langle Y\rangle $, is the complex 
\[
\dfrac{A\langle Y\rangle}{A+JY+(Y)^{(2)}}\cong\ldots \to SY_n\to SY_{n-1}\to \ldots \to SY_2\to SY_1\to 0
\]
where   $J=\Ker(A\langle Y\rangle \to  S)$ and $(Y)^{(2)}$ consists of the DG ideal of $A\langle Y\rangle $  consisting of all decomposable $\Gamma$-monomials on the set $Y$; see \cite[Construction~6.2.5]{IFR} for more details. Finally, $A\langle Y\rangle$ is an acyclic closure of $S$ if and only if $\ind_A^\gamma A\langle Y\rangle$ is minimal as a complex of (free) $S$-modules; cf.\@ \cite[Lemma~6.3.2]{IFR}. 
\end{chunk}
 We write $\shift $ for the shift functor on a graded object. Namely, for a graded object $V$, $\shift V$ is the graded object with
$(\shift V)_n=V_{n-1}$.

\begin{chunk}\label{c:dualmaps}
Given any minimal semifree  extensions  $k[X]$ and $k[X']$ over $k$, a local DG algebra map $\psi\colon  k[X]\to k[X']$ induces a well-defined map on the graded $k$-vector spaces 
\begin{equation}\label{e:indmap}
\ind_k\psi\colon  \ind_k k[X]\to \ind_kk[X']. 
\end{equation} 

Consider the commutative diagram of local rings 
\begin{center}
\begin{tikzcd}
    P \ar{r}\ar{d} & Q \ar{d}\\
    R\ar[r,"\vp"] & S
\end{tikzcd}
\end{center}
where the vertical maps are minimal Cohen presentations, with $\vp$ from above. Fix minimal models $P[X]\xra{\simeq}R$ and $Q[X']\xra{\simeq} S$, and let $\tilde{\vp}\colon P[X]\to Q[X']$ be the map induced by $\vp$; cf. \cite[Proposition~2.1.9]{IFR}. Set $\psi$ to be the map 
\[
k\otimes \tilde{\vp}\colon k[X]\to k[X'].
\]
By \cite[Theorem~3.4]{AlgebraRetracts}, there is the following commutative diagram of graded $k$-spaces:
\begin{center}
    \begin{tikzcd}
        (\shift\ind_k k[X'])^\vee \ar["(\shift\ind_k \psi)^\vee"]{rr} \ar["\cong"]{d} & &  (\shift\ind_k k[X])^\vee\ar["\cong"]{d}\\
    \pi^{\geqslant 2}(S) \ar["\pi^{\geqslant 2}(\vp)"]{rr} &   & \pi^{\geqslant 2}(R)
    \end{tikzcd}
\end{center}
where $(-)^\vee$ denotes graded $k$-space duality. In particular, we conclude that $\pi^{\geqslant 2}(\vp)$ is exactly the $k$-linear dual of $\ind_k\psi$, defined in \cref{e:indmap}, up to a shift.
\end{chunk}

\begin{chunk}\label{c:fibers}
We write $\F{\vp}$ for the (derived) fiber of $\vp$. Namely, $\F{\vp}$ is the DG $k$-algebra $k\otimes_R R[X]$ where $R[X]$ is a minimal model of $\vp$; it is clear that $\F{\vp}\simeq k\lotimes_R S,$ justifying the terminology. As any two minimal models are isomorphic, $\F{\vp}$ is an invariant of $\vp$.
In the construction above, if one replaces the minimal model $R[X]$ with an acyclic closure $R\langle Y\rangle$ for $\vp$ one obtains a semifree $\Gamma$-extension $k\otimes_R R\langle Y\rangle$, which will be denoted by $\F{\vp}_\gamma.$ 
\end{chunk}

\section{Tor-independent quotients}\label{sec:torind}
Throughout the rest of the article we fix the following notation. 
\begin{notation}\label{eq:MinimalIntersectionDiagramAS}
Consider the following diagram of surjective  local homomorphisms with common residue field $k$:
\begin{center}
\begin{tikzpicture}[baseline=(current  bounding  box.center)]
 \matrix (m) [matrix of math nodes,row sep=3em,column sep=4em,minimum width=2em] {
 R&S_1\\
 S_2&S\\};
 \path[->>] (m-1-1) edge node[above]{$\vp_1$} (m-1-2);
 \path[->>] (m-1-1) edge node[above]{$\vp$} (m-2-2);
 \path[->>] (m-1-1) edge node[left]{$\vp_2$} (m-2-1);
 \path[->>] (m-1-2) edge node[right]{$\psi_2$}(m-2-2);
 \path[->>] (m-2-1) edge node[below]{$\psi_1$} (m-2-2);
 \end{tikzpicture}
 \end{center} 
We denote by  $\m_R$ the maximal ideal of $R$ and assume that the kernel of $\vp_i$, denoted by $I_i$, is a  nontrivial proper ideal of $R$.
\end{notation}

\begin{remark} By a slight abuse of terminology, we will refer to the pair of maps $\vp_1,\vp_2$ as being Tor-independent provided their targets are Tor-independent over $R$. 
This is equivalent to requiring that $S$ is both the ordinary and derived pushout in the category of $R$-algebras  along the maps $\vp_1$ and $\vp_2$; the derived pushout condition simply says that $
S\simeq S_1\lotimes_R S_2.$
\end{remark}

\begin{chunk}\label{ch:history}
When $R$ is regular,  $\vp_1,\vp_2$ are  Tor-independent  if and only if $I_1I_2=I_1\cap I_2$. In this case, if each of the $I_i$ are contained in $\m_R^2$, the ring $S$ is  referred to as a minimal intersection ring  in  \cite{JM}. In \cite[Definition 5.5.1]{Vasconcelos}, the ideals $I_1$ and $I_2$  are called \emph{transversal}; such ideals, and their corresponding surjective ring maps, have been studied in \emph{loc.\@ cit.\@} as well as  \cite{tensor,small,Johnson,Keller}. 
\end{chunk}

\begin{example}\label{ex:ci}
Anytime $I_1$ is generated by an $R$-regular sequence that is also  $S_2$-regular  the pair $\vp_1,\vp_2$ are Tor-independent. In light of this point, Tor-independent maps generalize complete intersection maps of codimension at least two. 
\end{example}

\begin{lemma}\label{rem:DimCount}
Whenever $I_1\cap I_2\subseteq\m_R^2$, there is the equality
\[
\dim_k \dfrac{\m_S}{\m_S^2}=\dim_k \dfrac{\m_{S_1}}{\m_{S_1}^2}+\dim_k \dfrac{\m_{S_2}}{\m_{S_2}^2}-\dim_k \dfrac{\m_{R}}{\m_{R}^2}\,.
\]
 \end{lemma}
\begin{proof} Let $I$ be $I_1+I_2$ and  
\[
\bar\iota\colon \dfrac{I}{\m_RI}\longrightarrow\dfrac{\m_R}{\m_R^2}
\]
be the map induced by the inclusion $I \subseteq \m_R$; its codimension is exactly the embedding dimension of $S$, i.e., 
\begin{equation}\label{eq:linalg}
\dim_k \dfrac{\m_S}{\m_S^2}=\dim_k \dfrac{\m_R}{\m_R^2}-{\rm rank}\, \bar \iota\,.
\end{equation}
We define $\bar\iota_1$ and $\bar\iota_2$ similarly and note the analogous equalities to \cref{eq:linalg} also hold.

From the commutative diagram 
\begin{center}
    \begin{tikzcd}
    \dfrac{I}{\m_R I} \ar[r,"\bar\iota"] \ar[d,swap,"\cong"] & \dfrac{\m_R}{\m_R^2} \\
    \dfrac{I_1}{\m_R I_1} \oplus \dfrac{I_2}{\m_R I_2}  \ar[ru,swap,"\bar\iota_1+\bar\iota_2"]
    \end{tikzcd}
\end{center} and the fact that $\im \,\bar\iota_1$ and $\im\, \bar\iota_2$  intersect trivially in $\m_R/\m_R^2$, it follows that 
\begin{equation}\label{eq:additive}
    {\rm rank}\, \bar\iota={\rm rank}\, \bar\iota_1+{\rm rank}\, \bar\iota_2\,;
\end{equation}
 the vertical isomorphism in the diagram and the fact that $\im \,\bar\iota_1$ and $\im\, \bar\iota_2$  intersect at 0 are the only times the hypothesis $I_1\cap I_2\subseteq\m_R^2$ is used. Now substituting \cref{eq:additive} into \cref{eq:linalg} and using the analogs to \cref{eq:linalg} for  $\bar\iota_1$ and $\bar\iota_2$, the desired result follows readily. 
\end{proof}

The following lemma will be applied in  \cref{thm:SurjPiIsoAlmostSmall,thm:surjPiIsochar0}. It is likely well known but we include the proof here for convenience of the reader. The hypothesis of \cref{l:pi1} is satisfied if $\vp_1,\vp_2$ are Tor-independent, since $I_1\cap I_2=I_1I_2\subseteq \m_R^2$ is forced by the hypothesis that $\Tor^R_1(S_1,S_2)=0.$

\begin{lemma}
\label{l:pi1}
Adopting \cref{eq:MinimalIntersectionDiagramAS}, if $I_1\cap I_2\subseteq \m^2_R$, then there is the following isomorphism of $k$-vector spaces
\[
\pi^1(S) \cong \pi^1(S_1) \times_{\pi^1(R)} \pi^1(S_2)\,.
\]
\end{lemma}
\begin{proof}
Recall that for a local ring $(A,\m_A,k)$ there is the following natural isomorphism of $k$-spaces
\[
\pi^1(A)\cong \Hom_k\left(\frac{\m_A}{\m_A^2},k\right).
\]
Hence, by $k$-vector space duality it suffices to show the following diagram
\begin{equation*}
\begin{tikzpicture}[baseline=(current  bounding  box.center)]
 \matrix (m) [matrix of math nodes,row sep=3em,column sep=3.5em,minimum width=2em] {
 \dfrac{\m_R}{\m_R^2}&\dfrac{\m_{S_1}}{\m_{S_1}^2}\\
 \dfrac{\m_{S_2}}{\m_{S_2}^2}&\dfrac{\m_S}{\m_S^2}\\};
 \path[->>] (m-1-1) edge  node[above]{$\overline{\vp}_1$} (m-1-2);
 \path[->>] (m-1-1) edge  node[left]{$\overline{\vp}_2$}(m-2-1);
 \path[->>] (m-1-2) edge (m-2-2);
 \path[->>] (m-2-1) edge (m-2-2);
 \end{tikzpicture}
 \end{equation*}
is a pushout diagram of $k$-vector spaces. Let $P$ denote the pushout along $\overline{\vp}_1$ and $\overline{\vp}_2$. From the commutative diagram above the uniquely induced map $\psi\colon  P\to \m_S/\m_S^2$ is surjective; so it suffices to show their $k$-space dimensions coincide. 
From the assumption $I_1\cap I_2\subseteq \m_R^2$, it follows, by \Cref{rem:DimCount}, that 
\[
\dim_k \dfrac{\m_S}{\m_S^2}=\dim_k \dfrac{\m_{S_1}}{\m_{S_1}^2}+\dim_k \dfrac{\m_{S_2}}{\m_{S_2}^2}-\dim_k \dfrac{\m_{R}}{\m_{R}^2}.
\]

Moreover, 
\[
\dim_k P=\dim_k\dfrac{\m_{S_1}}{\m_{S_1}^2}+\dim_k \dfrac{\m_{S_2}}{\m_{S_2}^2}-\dim_k\left\{(\overline{\vp}_1(x),-\overline{\vp}_2(x))\colon   x\in \dfrac{\m_R}{\m_R^2}\right\}
\] and so to complete the proof it is enough to show 
\[
\dim_k\left\{(\overline{\vp}_1(x),-\overline{\vp}_2(x))\colon   x\in \dfrac{\m_R}{\m_R^2}\right\}=\dim_k\dfrac{\m_R}{\m_R^2}\,.
\]
The left-hand side is exactly the dimension of the image of the following surjective map
\[
\dfrac{\m_R}{\m_R^2}\xra{(\overline{\vp}_1,-\overline{\vp}_2)} \dfrac{\m_{S_1}}{\m_{S_1}^2}\oplus \dfrac{\m_{S_2}}{\m_{S_2}^2}\,;
\] note that $x$ is in the kernel of $(\overline{\vp}_1,-\overline{\vp}_2)$ if and only if $x=y+\m_R^2$ for some $y\in I_1\cap I_2$, and so the assumption that $I_1\cap I_2\subseteq \m_R^2$ implies $(\overline{\vp}_1,-\overline{\vp}_2$) is injective, as needed. 
\end{proof}

The rest of this section provides results  illustrating how  properties of $\vp$ determine and are determined by properties of $\vp_1$ and $\vp_2$, provided the latter pair are Tor-independent. We also show that in this setting $\vp_i$ and $\psi_i$ exhibit  properties simultaneously.

Recall the notions of  Golod, Gorenstein, and quasi-complete intersection homomorphisms discussed in \cite{Golod,AvramovFoxby,qci}, respectively, the last two generalizing the notions of Gorenstein and complete intersection local rings. Parts (1) and (2) of the next result generalize their ring theoretic versions in \cite[Proposition~3.3]{JM} with  simpler proofs not relying on the celebrated \emph{(new) intersection theorem}~\cite{PS,Roberts}.
\begin{proposition}\label{p:ci}
Suppose $\vp_i\colon R\to S_i$ is a pair of Tor-independent surjective local ring maps for $i=1,2$, and set $\vp=\vp_1\otimes \vp_2.$
\begin{enumerate}
    \item $\vp$ is Gorenstein if and only if each $\vp_i$ is;
    \item \label{p:ciCI} $\vp$ is  complete intersection if and only if each $\vp_i$ is;
    \item\label{p:ciQCI} more generally, $\vp$ is  quasi-complete intersection if and only if each $\vp_i$ is.
\end{enumerate}
\end{proposition}
\begin{proof} First, we point out that $S$ has finite projective dimension over $R$ if and only if each $S_i$ does since $\vp_1,\vp_2$ are Tor-independent. From this, \cref{p:ciCI} follows from \cref{p:ciQCI}; see \cite[Theorem~2.5]{qci}.

For the first statement, observe that there is the following isomorphism of DG $k$-algebras
\[
\F{\vp}\cong \F{\vp_1}\otimes_k \F{\vp_2}\,.
\]From this it follows that the homology of $\F{\vp}$ is finite dimensional if and only if each $\F{\vp_i}$ has finite dimensional homology. 
Now a calculation involving $k$-duality and the K\"{u}nneth formula yields the following isomorphism of graded $k$-spaces
\[
 \Ext_{\F{\vp}}(k,\F{\vp})\cong \Ext_{\F{\vp_1}}(k,\F{\vp_1})\otimes_k \Ext_{\F{\vp_2}}(k,\F{\vp_2})\,
\]  provided the homology of $\F{\vp}$ is finite. Therefore the desired result holds from \cite[Theorem~4.4]{AvramovFoxby}, and the observation at the beginning of the proof.

The Tor-independence of $\vp_1,\vp_2$ implies that given acyclic closures  $A_i\xra{\simeq}S_i$ for $\vp_i$, the DG $R$-algebra $A\colonequals A_1\otimes_R A_2$ is an acyclic closure for $\vp$.  As a direct consequence we obtain the following isomorphism of $k$-spaces 
\[
\ind_k^\gamma (A_1\otimes_Rk)\oplus \ind_k^\gamma (A_2\otimes_Rk)\cong \ind_k^\gamma (A\otimes_Rk)\,.
\]
Now (3) follows as the quasi-complete intersection property is characterized by the $k$-space above being concentrated in degrees one and two. 
\end{proof}

Whenever $\vp_1,\vp_2$ are Tor-independent, we have isomorphisms of DG $k$-algebras
$\F{\vp_i}\cong \F{\psi_i}$ for $i=1, 2$.  Thus, whenever a property of a  map can be characterized in terms of the fiber, the parallel map enjoys the same property. This is demonstrated in  \cref{p:parallelmaps} below.

\begin{proposition}\label{p:parallelmaps}
When $\vp_1,\vp_2$ are Tor-independent, with \cref{eq:MinimalIntersectionDiagramAS}, the following hold:
\begin{enumerate}
    \item $\vp_i$ is Golod if and only $\psi_i$ is Golod;
        \item $\vp_i$ is Gorenstein if and only if $\psi_i$ is Gorenstein;
    \item $\vp_i$ is (quasi-)complete intersection if and only $\psi_i$ is (quasi-)complete intersection.
    \item The map $\vp$ is never Golod.
\end{enumerate}
\end{proposition}
\begin{proof}
First, using that $\vp_1,\vp_2$ are Tor-independent,  the  isomorphism of DG $k$-algebras
$\F{\vp_i}\cong \F{\psi_i}$ yields one of graded Lie algebras over $k$
\[
\pi(\F{\vp_i}) \cong \pi(\F{\psi_i})\,.
\] Now (1) holds as the homotopy Lie algebra of a map being free characterizes the Golod property of that map; see \cite[Theorem~3.4]{Golod}.

The second assertion is immediate from $\F{\vp_i}\cong\F{\psi_i}$ and  \cite[Theorem 4.4]{AvramovFoxby}.

Finally, for (1) the assumption that $\vp_1,\vp_2$ are Tor-independent forces $A\otimes_R k$ to be isomorphic to $B\otimes_{S_j}k$ where $A$ and $B$ are acyclic closures over $\vp_i$ and $\psi_i$, respectively, where $i\neq j$. This determines the isomorphism of $k$-spaces
\[
\ind_k^\gamma (A\otimes_R k)\cong \ind_k^\gamma (B\otimes_{S_j}k)\,.
\] Now the desired result holds as the quasi-complete intersection property is detected by these $k$-spaces being concentrated in degrees one and two. The statement for complete intersection homomorphisms now follows using that $S_i$ has finite projective dimension over $R$ if and only if  $S$ has finite projective dimension over $S_j$, since $\vp_1,\vp_2$ are Tor-independent; again one is also appealing to \cite[Theorem~2.5]{qci}.

For (4), recall \cite[Theorem 3.4(4)]{Golod} which asserts a map $\phi$ is Golod if and only if $\pi(\F{\phi}_\gamma)$ is a free graded Lie algebras. It is a direct calculation that
\[
\pi(\F{\vp}_\gamma)\cong \pi(\F{\vp_1}_\gamma)\times \pi(\F{\vp_2}_\gamma)
\] 
and it is well-known a product of nonzero Lie algebras cannot be free, the desired result follows; the fact that each $R\to S_i$ has nontrivial kernel, forces the Lie algebras in question to be nonzero.
\end{proof}

We end this section by abandoning the local situation of \cref{eq:MinimalIntersectionDiagramAS}, and instead work in the grading setting. 
\begin{chunk}
\label{chunk:graded}
Let $R$ be a standard graded connected $k$-algebra. One can adapt the definitions from \cref{sec:DG} by requiring all cycles, ideals, maps, etc. to be homogeneous with respect to the internal degree determined by the grading on $R$. Furthermore, in this setting, $\pi(S)$ is naturally bigraded with respect to the usual cohomological degree, as well as the internal degree acquired from $R$; hence we will write $\pi^{*,*}(S)$ to indicate that the homotopy Lie algebra of $S$ is being considered as a bigraded object. Finally, recall that $S$ is a Koszul algebra if $k$ has a linear minimal free resolution over $S$; this resolution being the acyclic closure forces
$\pi^{i,j}(S)=0$ for all $i\neq j$ with $i\geqslant2$.
\end{chunk}

\begin{proposition}\label{p:Koszul}
Let $R$ be a standard graded polynomial ring over the field $k$ and consider quotients $S_1,S_2$ by homogeneous ideals of $R$. If $S_1,S_2$ are  Tor-independent over $R$, then $S$ is Koszul if and only if each $S_i$ is Koszul. 
\end{proposition}
\begin{proof}
In the graded setting, as in the proof of \cref{p:ci}, it follows directly from the assumption that $S_1,S_2$ are Tor-independent that $A\colonequals  A_1\otimes_R A_2$ is a minimal model for $S$ where $A_i$ is a minimal model for $S_i.$
As a consequence there is an isomorphism of bigraded vector spaces 
\[
\ind_k (A_1\otimes_Rk)\oplus \ind_k (A_2\otimes_Rk)\cong \ind_k (A\otimes_Rk)\,.
\] Shifting and taking $k$-linear dual induces an isomorphism of bigraded $k$-spaces \[\pi^{\geqslant 2,*}(S)\cong\pi^{\geqslant 2,*}(S_1)\times\pi^{\geqslant 2,*}(S_2)\,;\]this is analogous to the local case explained in  \cref{c:dualmaps}. The desired result now follows immediately from \cref{chunk:graded}.
\end{proof}

\section{Almost small homomorphisms}\label{sec:AS}

In this section we still adopt \cref{eq:MinimalIntersectionDiagramAS} and prove the conclusion of \cref{introthm1} holds when at least one  $\vp_i$  is almost small.

\begin{chunk}\label{c:almostsmall}
A surjective local map $\vp\colon R\to S$ is \emph{almost small} if  the morphism of graded Lie algebras
\[
\pi^{\geqslant 2}(\vp)\colon   \pi^{\geqslant 2}(S)\to \pi^{\geqslant 2}(R)
\] is surjective. In \cite{AlgebraRetracts} these are defined in terms of Tor algebras. Namely, the map on Tor algebras 
\[
\Tor^\vp(k,k)\colon   \Tor^R(k,k)\to \Tor^S(k,k)
\] has its kernel generated in degree one; these definitions coincide by  \cite[Proposition 4.3]{AlgebraRetracts}. 
\end{chunk}

We first prove a result that identifies pullbacks, also known as fiber products, of graded Lie algebras.

\begin{lemma}\label{pullbacklemma}
Given the following commutative diagram of graded Lie $k$-algebras
\begin{equation*}
\begin{tikzpicture}[baseline=(current  bounding  box.center)]
 \matrix (m) [matrix of math nodes,row sep=2.5em,column sep=3em,minimum width=2em] {
 \g&\h_1\\
 \h_2&\f.\\};
 \path[->>] (m-2-2) edge node[below]{$\gamma$} (m-2-1);
 \path[->] (m-2-2) edge node[right]{$\alpha$} (m-1-2);
 \path[->] (m-1-2) edge node[above]{$\beta$} (m-1-1);
 \path[->] (m-2-1) edge node[left]{$\delta$} (m-1-1);
 \end{tikzpicture}
 \end{equation*}
If $ \alpha|_{\Ker\gamma}\colon \Ker\gamma\rightarrow \Ker \beta$ is an isomorphism of Lie ideals, then $\f\cong \h_1\times_{\g}\h_2.$
\end{lemma}
\begin{proof}
The commutative diagram defines a graded Lie $k$-algebra map $
\sigma\colon \f\to \h_1\times_{\g}\h_2$ given by $x\mapsto (\alpha(x),\gamma(x)).$ We show $\sigma$ is an isomorphism. 

Let $x\in \Ker \sigma,$ then $\alpha(x)=0$ and $\gamma(x)=0$. The latter says $x\in \Ker \gamma$ and so the former implies $x=0$ since $ \alpha|_{\Ker\gamma}$ is injective by assumption; thus, $\sigma$ is injective.

Let $(x_1,x_2)\in \h_1\times_{\g}\h_2,$ then by the assumption that $\gamma$ is surjective there exists $y\in \f$ such that $\gamma(y)=x_2.$ As $\delta(x_2)=\beta(x_1)$ and the diagram above commutes  we conclude that $\alpha(y)-x_1\in \Ker \beta$. Now using that $\alpha(\Ker\gamma)=\Ker \beta,$ there exists $y'\in \Ker \gamma$ such that $\alpha(y')=\alpha(y)-x_1.$ Now observe that 
\[
\sigma(y-y')=(\alpha(y-y'),\gamma(y-y'))=(x_1,\gamma(y))=(x_1,x_2),
\] as desired, justifying that $\sigma$ is surjective. 
\end{proof}

\begin{theorem}\label{thm:SurjPiIsoAlmostSmall}
Suppose $S_1\xla{\vp_1}R\xra{\vp_2} S_2$ is a pair of Tor-independent surjective local ring maps with residue field $k$, and set $S=S_1\otimes_R S_2$. If at least one $\vp_i$ is almost small, then $\pi$ induces the following isomorphism of graded Lie $k$-algebras \[\pi(S)\cong \pi(S_1)\times_{\pi(R)}\pi(S_2)\,.\]
\end{theorem}
\begin{proof}
We assume that $\vp_1$ is almost small. As homotopy Lie algebras and the Tor-independence of the maps are both invariant under completion, see \cref{HomotopyLieAlgebra} for the former point, we can assume $R$ is complete. Now let $\rho\colon P\rightarrow R$ be a minimal Cohen presentation of $R$.

 Let $P[W]\to R$ be a minimal model of $\rho$. Since $\vp_1$ is almost small, by \cite[Theorem 4.11]{AlgebraRetracts}, we have the following commutative diagram of differential graded $P$-algebras: 
\begin{equation*}
\begin{tikzpicture}[baseline=(current  bounding  box.center)]
 \matrix (m) [matrix of math nodes,row sep=2.5em,column sep=3em,minimum width=2em] {
 P[W]&Q_1[X,W]\\
 R&S_1\\};
 \path[->] (m-1-1) edge node[above]{$\tilde{\vp}_1$} (m-1-2);
 \path[->] (m-1-1) edge  (m-2-1);
 \path[->] (m-1-2) edge (m-2-2);
 \path[->] (m-2-1) edge node[above]{$\vp_1$} (m-2-2);
 \end{tikzpicture}
 \end{equation*}
where $Q_1$ is a regular quotient of $P$ with $Q_1\to S_1$ a minimal Cohen presentation; the vertical map on the right $Q_1[X, W] \to S_1$ is a minimal model for $S_1$ over $Q_1$, and $\tilde{\vp}_1(w)=w$ for all $w\in W$. Let $Q_2[Y] \to S_2$ be minimal model of $S_2$ where $Q_2$ is a regular quotient of $P$ and  $Q_2\to S_2$ is a minimal Cohen presentation. Finally, set $Q$ to be the regular quotient $Q_1\otimes_P Q_2$ of $P$.

Since $\vp_1,\vp_2$ are Tor-independent, the canonical map 
\[
Q_1[X, W] \otimes_{P[W]} Q_2[Y] \to S
\]
is a quasi-isomorphism where $Q_2[Y]$ is regarded as a DG $P[W]$-module by lifting $\vp_2$. The following isomorphisms of graded algebras \begin{align*}
  Q_1[X, W] \otimes_{P[W]} Q_2[Y] & \cong Q_1[X] \otimes_{P} \left(P[W] \otimes_{P[W]} Q_2[Y]\right)\\
    &\cong Q[X, Y]
\end{align*}
induce a differential on the semifree extension $Q[X, Y]$. Moreover, $Q[X, Y] \to S$ is a minimal model over $Q$.
Indeed, the isomorphisms above induce an isomorphism
\[
\alpha\colon k[X, W] \otimes_{k[W]} k[Y] \xra{\cong} k[X, Y]
\]
where $\alpha(x \otimes 1) = x$ and $\alpha(1 \otimes y) = y$ for each $x \in X, y \in Y$, and  $\alpha(W \otimes 1)$ is contained in the DG ideal generated by $Y$ in $k[X, Y]$.  As a consequence the induced differential  on $k[X,Y]$ satisfies  
\[
\del(k[X, Y]) \subseteq (X, Y)^2
\] since $k[X,W]$ and $k[Y]$ are minimal semifree extensions over $k$. Finally, 
\[
Q\cong Q_1\otimes_P Q_2\to S_1\otimes_P S_2\cong S
\] is a minimal Cohen presentation as $Q_i\to S_i$ are each minimal Cohen presentations. 
Thus, we have the following commutative diagram of semifree extensions
\begin{equation}\label{eq:minmodels}
\begin{tikzpicture}[baseline=(current  bounding  box.center)]
 \matrix (m) [matrix of math nodes,row sep=2.5em,column sep=3em,minimum width=2em] {
 k[W]&k[X,W]\\
 k[Y]&k[X,Y]\\};
 \path[right hook->] (m-1-1) edge  (m-1-2);
 \path[->] (m-1-1) edge  (m-2-1);
 \path[->] (m-1-2) edge (m-2-2);
 \path[right hook->] (m-2-1) edge (m-2-2);
 \end{tikzpicture}
 \end{equation}
where the horizontal maps are the canonical inclusions and the vertical map on the right being the identity when restricted to $X$. 

Since the minimal models were taken with respect to \emph{minimal} Cohen presentations, by \cref{c:dualmaps} there exists the following corresponding commutative diagram of graded Lie algebras:
\begin{equation}\label{eq:liealgebras}
\begin{tikzpicture}[baseline=(current  bounding  box.center)]
 \matrix (m) [matrix of math nodes,row sep=3em,column sep=4em,minimum width=2em] {
 \pi^{\geqslant2}(R)&\pi^{\geqslant2}(S_1)\\
 \pi^{\geqslant2}(S_2)&\pi^{\geqslant2}(S)\\};
 \path[->>] (m-1-2) edge node[above]{$\pi^{\geqslant 2}(\vp_1)$} (m-1-1);
 \path[->]  (m-2-1) edge node[left]{$\pi^{\geqslant 2}(\vp_2)$} (m-1-1);
 \path[->]  (m-2-2) edge node[right]{$\pi^{\geqslant 2}(\psi_2)$} (m-1-2);
 \path[->>]  (m-2-2) edge node[below]{$\pi^{\geqslant 2}(\psi_1)$} (m-2-1);
 \end{tikzpicture}
 \end{equation}
where the horizontal maps in \cref{eq:liealgebras} are natural projections whose kernels have a $k$-basis of derivations indexed by the variables $X$. Also, as the vertical map on the right in \cref{eq:minmodels} is the identity when restricted to  $X$, the morphism $\pi^{\geqslant2}(\psi_2)$ maps the derivations corresponding to $X$ in $\pi^{\geqslant 2}(S)$ bijectively to the  derivations corresponding to $X$ in $\pi^{\geqslant 2}(S_1)$. Applying \cref{pullbacklemma} we have established the following isomorphism of graded Lie algebras
\begin{equation}\label{eq:iso1}
\pi^{\geqslant 2}(S)\cong \pi^{\geqslant 2}(S_1)\times_{\pi^{\geqslant 2}(R)}\pi^{\geqslant 2}(S_2)\,.\end{equation}
From \cref{l:pi1}, we obtain the following isomorphism of $k$-spaces
\begin{equation}\label{eq:iso2}
\pi^1(S) \cong \pi^1(S_1) \times_{\pi^1(R)} \pi^1(S_2)\,.
\end{equation}
The isomorphisms in \cref{eq:iso1} and \cref{eq:iso2} are  restrictions of the unique map of graded Lie algebras 
\[ 
\pi(S) \to \pi(S_1) \times_{\pi(R)} \pi(S_2)\,,
 \] completing the proof of the theorem. 
\end{proof}

\begin{remark}\label{remark:cmonman}
With the stronger assumption that one of the $\vp_i$ is assumed to be small, the assertion of \cref{thm:SurjPiIsoAlmostSmall} was established in \cite{small}. The argument in \emph{loc.\@ cit.\@} proceeds by establishing a spectral sequence of  Hopf algebras whose $\textrm{E}^2$-page is determined by the Tor Hopf algebras over $R$ and the $S_i;$ the $\textrm{E}^\infty$-page of this spectral sequence is the Tor algebra over $S$. 
Tor-independence of the $S_i$ over $R$, forces the spectral sequence to degenerate from which the isomorphism asserted in \cref{cor:TorAlg} can be established. 
Finally, one can now use a celebrated equivalence of categories of Andr\'{e}, Milnor--Moore, and Sj\"{o}din in \cite{Andre,MM,Sjodin} to prove the assertion of \cref{thm:SurjPiIsoAlmostSmall} holds---still with the assumption that one of the $\vp_i$ is assumed to be small.

It seems likely that one can follow a similar line of arguing to prove \cref{thm:SurjPiIsoAlmostSmall}; however, besides involving a wealth of machinery in \cite{tensor,small}, there are still a few sticking points we wish to highlight. A hurdle is establishing an analog of the spectral sequence in \cite[Theorem 5.5]{small}. One can use \cite[Theorem 4.11]{AlgebraRetracts} in lieu of \cite[Corollary 5.4]{small} when trying to extend \cite[Theorem 5.5]{small} to when one of the $\vp_i$ is small. However it is not clear  to the  authors how to adapt \cite[Lemma 5.2]{small}, and the references therein, to work in this generality. 
\end{remark}
\begin{remark}
When $R$ is regular the assumption that one of the $\vp_i$ is almost small is vacuous. In particular, the conclusion of \cref{thm:SurjPiIsoAlmostSmall} holds for any  minimal intersection ring in the terminology of \cite{JM}; cf.\@ \cref{ch:history}. As a consequence, such a ring $S$ whose completion $\widehat{S}$ can be written as $Q/(J_1+J_2)$ where $Q$ is a regular local ring and $J_i$  are nontrivial ideals with $\Tor^Q(Q/J_1,Q/J_2)\cong \widehat{S}$,
has $\pi^{\geqslant 2}(S)$ decomposable as 
\[
\pi^{\geqslant 2}(Q/J_1)\times \pi^{\geqslant 2}(Q/J_2)\,
\] where we use the isomorphism $\pi^{\geqslant 2}(S)\cong \pi^{\geqslant 2}(\widehat{S})$; see \cref{HomotopyLieAlgebra}. This observation can be used to deduce the well known fact that the homotopy Lie algebra of a complete intersection ring of codimension $c\geqslant 2$ in degrees strictly larger than one is an abelian Lie algebra with $c$ generators in degree two; cf.\@ \cref{ex:ci}.
\end{remark}

\begin{remark}
Let $\vp\colon R\rightarrow S$ be a not necessarily surjective map of local rings. Take a Cohen factorization of (the completion of) $\vp$
\[
 \widehat R\xra{\dot{\vp}}R'\xra{\vp'}\widehat{S},
\] as described in \cite{AFH};
here $\dot{\vp}$ is weakly regular---in the sense it is a flat local homomorphism with regular fiber---and  $\vp'$ is surjective.
If $\widehat{S}$ can be written as $S_1\otimes_{R'}S_2\cong\widehat S$, where $S_1, S_2$ are Tor-independent quotients of $R$ and at least one of $R'\to S_i$ is almost small, then from \Cref{thm:SurjPiIsoAlmostSmall} one can deduce\[
\pi^{\geqslant2}(S)\cong\pi^{\geqslant2}(S_1)\times_{\pi^{\geqslant2}(R)}\pi^{\geqslant2}(S_2),
\]
using the facts that $\pi^{\geqslant 2}(R)\cong\pi^{\geqslant2}(R')$ since $\dot{\vp}$ is weakly regular and $\pi^{\geqslant 2}(S)\cong \pi^{\geqslant 2}(\widehat{S})$.

If $\vp$ is assumed to be almost small to begin with then, by \cite[Proposition 4.8]{AlgebraRetracts}, the map $\vp'$ is also almost small. It follows from \cite[Corollary 4.5(b)]{AlgebraRetracts}, that when $\widehat{S}$ is a tensor product of Tor-independent quotients $S_1,S_2$ of $R'$ that both maps $R'\to S_i$ are always almost small.
\end{remark}

\section{Residual characteristic zero and finite weak category}
\label{sec:char0}
We keep the notation set in \cref{eq:MinimalIntersectionDiagramAS}.


\begin{chunk}\label{c:AdjoinVariables}
Given a surjective quasi-isomorphism of  DG algebras $\alpha\colon A\to B$ we recall a lifting property of semifree extensions. For a set of cycles $Z$ in $B$, upon choosing a set of lifts $\tilde{Z}$ for $Z$ in $A$, there exists a surjective quasi-isomorphism 
\[
A[X]\to B[X]
\] extending $\alpha$ and sending each $x$ to the corresponding $x$ where $X=\{x_z: |x_z|=|z|+1\}_{z\in Z}$ and 
\[
\del^{A[X]}(x_z)=\tilde{z} \ \text{ and } \ \del^{B[X]}(x_z)=z\,.
\]

This will be applied specifically in the following setting:
\[
Q\xra{\rho} R\xra{\vp} S
\] is a composition of surjective local maps with minimal models $Q[W]$ and $R[X]$ for $\rho$ and $\vp,$ respectively. The surjective quasi-isomorphism $Q[W]\xra{\simeq}R$ extends to a surjective quasi-isomorphism 
\[
Q[W,X]\xra{\simeq} R[X]
\]as described above. Note as the map above is a quasi-isomorphism, we have $Q[W,X]\xra{\simeq}S$. We will freely use this construction in the sequel. 
\end{chunk}

 Below we recall a concept from \cite{AlgebraRetracts}. 
\begin{chunk}\label{d:WeakCat}
A surjective map $R\xra{\vp}S$ of local rings is said to have \emph{finite weak category} if there exists $s\geqslant 1$ such that for each $n\geqslant 2$ the product of $s$ elements from $\hh_{>0}(k[X_{\geqslant n}])$ is zero where $R[X]\xra{\simeq} S$ is a minimal model for $\vp$.
\end{chunk}

\begin{chunk}\label{c:WeakCatGMin}
By \cite[Theorem~5.6]{AlgebraRetracts}, maps of finite projective dimension have finite weak category. Note that the \emph{weak category} of a map is a quantity defined in \emph{loc.\@ cit}. However for the purposes of the present article we only require it be finite and can disregard its actual value.  
\end{chunk}

\begin{lemma}\label{l:WeakCat}
Let $\vp\colon R\to S$ be a surjective map of local rings and let $P\to \widehat{R}$ be a minimal Cohen presentation. If $\vp$ has finite weak category or $k$ has characteristic zero then there is a commutative diagram
\begin{center}
\begin{tikzpicture}[baseline=(current  bounding  box.center)]
 \matrix (m) [matrix of math nodes,row sep=3em,column sep=4em,minimum width=2em] {
P[W]&Q[U,X]\\
\widehat{R}&\widehat{S}\\};
 \path[->] (m-1-1) edge node[above]{$\tilde\vp$} (m-1-2);
  \path[->] (m-1-2) edge node[right]{$\simeq$} (m-2-2);
 \path[->] (m-1-1) edge node[left]{$\simeq$} (m-2-1);
  \path[->] (m-2-1) edge  (m-2-2);
 \end{tikzpicture}
\end{center}
satisfying the following conditions: 
\begin{enumerate}
\item the vertical maps are minimal models with respect to minimal Cohen presentations $P\to \widehat{R}$ and $Q\to \widehat{S}$;
\item $\tilde{\vp}$ extends a surjective map $P\to Q$, with $U\subseteq W$ and $\tilde{\vp}(w)=\begin{cases}
    w & \text{if }w\in U\\
    0 & \text{otherwise}
    \end{cases}$
    for $w\in W.$
\end{enumerate}
\end{lemma}
\begin{proof}
Assume $R$ and $S$ are complete, and fix minimal models $P[W]\to R$ and $R[Y]\to S$.
In the case that $\vp$ has finite weak category we can assume 
$P[W, Y]$, as constructed in \cref{c:AdjoinVariables},  satisfies 
\begin{equation}\label{eq:GulliksenMinimal}
\del_{2i}(P[W, Y]) \subseteq \m_P[W, Y] + (W, Y)^2
\end{equation} for all $i\geqslant 0$, equivalently $\del_{2i}(\ind_kk[W,Y]) = 0$ for all 
$i\geqslant 0$; see \cite[Theorem~25]{Briggs}. 
Now under either hypothesis,  there exists a minimal semifree extension $P[U,X']$ with $U\subseteq W$ and $X'\subseteq Y$ that fits into a commutative diagram 
\begin{center}
\begin{tikzpicture}[baseline=(current  bounding  box.center)]
 \matrix (m) [matrix of math nodes,row sep=3em,column sep=4em,minimum width=2em] {
P[W,Y]&P[U,X']\\
&S\\};
 \path[->] (m-1-1) edge node[above]{$\simeq$} (m-1-2);
 \path[->] (m-1-2) edge node[right]{$\simeq$} (m-2-2);
 \path[->] (m-1-1) edge node[below]{$\simeq$} (m-2-2);
 \end{tikzpicture}
 \end{center}
where the horizontal map is determined by $t\mapsto \begin{cases}
t & t\in U\cup X' \\
0 & \text{otherwise}
\end{cases}$
for $t\in W\cup Y$ and $X_{2i+1}'=Y_{2i+1}$ for all nonnegative integers $i$. 
The existence of this diagrams is contained in the proof of \cite[Lemma~3.2.1]{GL} in the case that the characteristic  of $k$  is zero,  and it holds when \cref{eq:GulliksenMinimal} is satisfied by \cite[Theorem~24]{Briggs}. 

 Next take the ideal $J$ in $P$ generated by the elements in
 \[(\m_P\backslash\m_P^2)\cap \ker(P\to S)\]
 and set $Q=P/J;$ note that by construction $Q\to S$ is a minimal Cohen presentation. The quotient map $P\to Q$ induces the quasi-isomorphism of DG algebras
 $P[X_1']\xra{\simeq}  Q[X_1]$ where $X_1'\backslash X_1$ corresponds to a minimal generating set for $J$ and for $x'\in X_1'$ we have
 \[
 x'\mapsto \begin{cases}
 x' & x'\notin X_1\\
 0 & x'\in X_1\,.
 \end{cases}
 \]
 By \cref{c:AdjoinVariables}, this extends to a quasi-isomorphism of DG algebras 
 $ P[U,X']\xra{\simeq} Q[U,X]$ with $X_i=X_i'$ for $i>1$ satisfying $x\mapsto x$ for each $x\in X_{\geqslant2}.$
 Setting $\tilde{\vp}$ to be the composition of the maps 
 \[
 P[W]\to P[W,Y]\xra{\simeq} P[U,X']\xra{\simeq} Q[U,X]\,,
 \]
where the maps are constructed above, has the desired properties.
Finally, by the construction of $P[U,X']\xra{\simeq} Q[U,X]$ it follows that $Q[U,X]$ is a minimal model for $S$ over $Q$ since $P[U,X']$ is a minimal model for $S$ over $P$.
\end{proof}

The main result of this section shows the conclusion of \cref{introthm1} is satisfied when the residue field has characteristic zero or when each $\vp_i$ has finite weak category.  

\begin{theorem}
\label{thm:surjPiIsochar0}
Let $R\xra{\vp_i} S_i$ be  surjective ring maps with residue field $k$ for $i=1,2$, and  set $S=S_1\otimes_R S_2.$
If  $\vp_1,\vp_2$ are Tor-independent, and in addition the characteristic of $k$ is zero or each $\varphi_i$ has finite weak category, then the functor $\pi$ induces the following isomorphism of graded Lie algebras \[\pi(S)\cong \pi(S_1)\times_{\pi(R)}\pi(S_2)\,.\]
\end{theorem}

\begin{proof}
Without loss of generality we can assume $R$ is complete, and let $P\to R$ be a minimal Cohen presentation. By \cref{l:WeakCat},  there exist minimal semifree  extensions $Q_1[U, X]$ and $Q_2[V,Y]$ where
 $U, V \subseteq W$, that fit into commutative diagrams
\begin{center}
\begin{tikzpicture}[baseline=(current  bounding  box.center)]
 \matrix (m) [matrix of math nodes,row sep=3em,column sep=4em,minimum width=2em] {
P[W]&Q_1[U,X]\\
&S_1\\};
 \path[->] (m-1-1) edge node[above]{ } (m-1-2);
 \path[->] (m-1-2) edge node[right]{ } (m-2-2);
 \path[->] (m-1-1) edge node[below]{$\simeq$} (m-2-2);
 \end{tikzpicture}\quad
 \begin{tikzpicture}[baseline=(current  bounding  box.center)]
 \matrix (m) [matrix of math nodes,row sep=3em,column sep=4em,minimum width=2em] {
P[W]&Q_2[V,Y]\\
&S_2\\};
 \path[->] (m-1-1) edge node[above]{ } (m-1-2);
 \path[->] (m-1-2) edge node[right]{ } (m-2-2);
  \path[->] (m-1-1) edge node[below]{$\simeq$} (m-2-2);
 \end{tikzpicture}
 \end{center} 
where the horizontal maps are the DG algebra maps determined by 
\[
w\mapsto \begin{cases}
w & w\in U \\
0 & w\notin U
\end{cases} \ \text{ and } \ w\mapsto \begin{cases}
w & w\in V \\
0 & w\notin V\,,
\end{cases}
\] respectively, and the vertical maps are the augmentation maps. Because $\vp_1, \vp_2$ are Tor-independent we have the following quasi-isomorphism 
\[
Q_1[U, X] \otimes_{P[W]} Q_2[V, Y] \xra{\simeq} S\,.
\]
Furthermore, as the kernels of the maps $P[W] \to Q_1[U, X]$ and $P[W] \to Q_2[V, Y]$ are each generated by a part of a minimal generating set for $\m_P$ together with a subset of $W$, it follows that the DG algebra $Q_1[U, X] \otimes_{P[W]} Q_2[V, Y]$ is a semifree  extension of $Q=Q_1\otimes_P Q_2$. Finally, since $Q_1[U, X]$ and $Q_2[V, Y]$ are minimal, we conclude  that $Q_1[U, X] \otimes_{P[W]} Q_2[V, Y]$ is a minimal model of $S$ over $Q$.

Therefore, taking indecomposables induces the following commutative diagram:
\begin{equation}\label{eq:pushout}
\begin{tikzpicture}[baseline=(current  bounding  box.center)]
 \matrix (m) [matrix of math nodes,row sep=3em,column sep=4em,minimum width=2em] {
kW&kU\oplus kX\\
kV\oplus kY&(kU\cap kV)\oplus kX\oplus kY\\};
 \path[->] (m-1-1) edge  (m-1-2);
 \path[->] (m-1-2) edge  (m-2-2);
 \path[->] (m-1-1) edge (m-2-1);
  \path[->] (m-2-1) edge (m-2-2);
 \end{tikzpicture}
\end{equation}
where maps are the identity when restricted to $kX$ or $kY$ and the obvious projection onto the first component when restricted to $kW$, $kU$, or $kV$. Following \cref{c:dualmaps}, upon shifting and taking $k$-linear duals one obtains the following isomorphism of Lie algebras, since \cref{eq:pushout} is a pushout diagram of $k$-vector spaces: 
\begin{equation}\label{e:isofiber}
\pi^{\geqslant{2}}(S)\cong \pi^{\geqslant2}(S_1)\times_{\pi^{\geqslant 2}(R)}\pi^{\geqslant 2}(S_2)\,.
\end{equation} Again using that $\vp_1, \vp_2$ are Tor-independent, we have that $I_1 \cap I_2 = I_1I_2$. Therefore one can apply \cref{l:pi1}, finishing the proof of the desired result when combined with \cref{e:isofiber}.
\end{proof}

By  \cite[Theorem~5.6]{AlgebraRetracts}, if $\vp_i$ is almost small, then it has finite weak category. So in light of this point and  \cref{thm:surjPiIsochar0}, the next question is prompted:
\begin{question}
Can one replace the assumption that at least one of $\vp_i$ is almost small with the assumption that at least one $\vp_i$ has finite weak category in \cref{thm:SurjPiIsoAlmostSmall}?
\end{question}

\begin{remark}
Let $\vp\colon R\rightarrow S$ be a not necessarily surjective map of local rings. Consider a Cohen factorization of (the completion of) $\vp$ 
\[
\widehat R\xra{\dot{\vp}}R'\xra{\vp'}\widehat{S}\,.
\]
Assume there is an isomorphism $S_1\otimes_{R'}S_2\cong\widehat S$ where $S_1,S_2$ are Tor-independent quotients of $R'$, and that one of the following hold:
\begin{enumerate}
\item The characteristic of the residue field of $S$ is zero;

\item The maps $R'\to S_i$ have finite weak category for $i=1,2$.
\end{enumerate}
From \cref{HomotopyLieAlgebra} and \Cref{thm:surjPiIsochar0} it follows that 
\[
\pi^{\geqslant2}(S)\cong\pi^{\geqslant2}(S_1)\times_{\pi^{\geqslant2}(R)}\pi^{\geqslant2}(S_2).
\]
\end{remark}

\section{Applications}\label{sec:applications}
In this section we list several consequences of \cref{thm:SurjPiIsoAlmostSmall,thm:surjPiIsochar0}. The first corollary, which is \cref{introthm2} from the introduction, recovers \cite[Corollary~5.6(b)]{small} when one of the Tor-independent maps is small. First, a bit of notation. 

\begin{chunk}\label{c:deviations}For a local ring $A$ with residue field $k$, recall the Poincar\'{e} series of $k$ over $A$ is the formal power series
\[
 \mathrm{P}^A_k(t)= \sum_{i\geqslant 0} \dim_k (\Ext_A^i(k,k))t^i\,.
\]The number $\ve_i(A)\colonequals \dim_k \pi^{i}(A)$ is called the $i^{\text{th}}$ deviation of $A$. The sequence of deviations are encoded in the Poincar\'e series  of $k$ over $A$ according to the equality
\[\mathrm{P}^A_k(t) = \prod_{i=1}^\infty\frac{(1+t^{2i-1})^{\ve_{2i-1}(A)}}{(1-t^{2i})^{\ve_{2i}(A)}}\,;\]  for this equality see \cite[7.1]{IFR}
\end{chunk}

\begin{corollary}\label{cor:PSeriesAS}
Suppose $S_1\xla{\vp_1}R\xra{\vp_2} S_2$ is a pair of Tor-independent surjective local ring maps with residue field $k$, and set $S=S_1\otimes_R S_2$. If at least one $\vp_i$ is almost small, then 
\[
\mathrm{P}^{S}_k(t)=\frac{\mathrm{P}^{S_1}_k(t)\cdot \mathrm{P}^{S_2}_k(t)}{\mathrm{P}^R_k(t)}.
\]
\end{corollary}
\begin{proof}
By \cref{thm:SurjPiIsoAlmostSmall}, there is the following exact sequence of graded Lie algebras over $k$:
\[
0\to \pi^{\geqslant2}(S_1)\times_{\pi^{\geqslant2}(R)}\pi^{\geqslant2}(S_2)\to \pi^{\geqslant2}(S_1)\times \pi^{\geqslant2}(S_2)\to \pi^{\geqslant2}(R)
\] where the map on the right is the difference of the induced maps on homotopy Lie algebras. Under the assumptions of the corollary, we deduce that the map on the right is surjective and therefore
\[
\varepsilon_i(S)=\varepsilon_i(S_2)+\varepsilon_i(S_2)-\varepsilon_i(R)
\] for all $i\geqslant 2.$ Combining this with \cref{l:pi1} we obtain the desired result on Poincar\'{e} series. 
\end{proof}
Notice the conclusion of \cref{cor:PSeriesAS} is only asserted under one of the  three conditions in \cref{introthm1}, while some of the following corollaries (for example,  \cref{cor:TorAlg}) hold under any of the conditions listed in \cref{introthm1}. In light of this point we pose the following question. 
\begin{question}\label{question}
In the notation of  \cref{cor:PSeriesAS}, does the conclusion of \cref{cor:PSeriesAS} hold if the characteristic of $k$ is zero, or if both $\vp_1,\vp_2$ have finite weak category?
\end{question}

\begin{remark}
The  conclusion of \cref{cor:PSeriesAS} also holds if at least one of the maps $\psi_i$ is large, under the standing assumption that $S_1,S_2$ are Tor-independent.

Indeed, if $\psi_1$ is large, then the induced map
$\Tor^{S_1}(S,k) \to \Tor^S(k,k)$ is injective; see \cite[Theorem 1.1]{levin}. 
Furthermore, from following commutative diagram 
\begin{equation*}
\begin{tikzpicture}[baseline=(current  bounding  box.center)]
 \matrix (m) [matrix of math nodes,row sep=3em,column sep=4em,minimum width=3em] {
\Tor^R(S_2,k)&\Tor^{S_1}(S,k)\\
\Tor^R(k,k)&\Tor^{S_1}(k,k),\\};
\path[->] (m-1-1) edge  node[above]{$\cong$} (m-1-2);
\path[->] (m-1-1) edge (m-2-1);
\path[->] (m-1-2) edge  (m-2-2);
\path[->] (m-2-1) edge (m-2-2);
\end{tikzpicture}
\end{equation*}
it follows that the map $\Tor^R(S_2,k) \to \Tor^R(k,k)$ is injective. Therefore $\vp_1$ is also large.  Finally by \cite[Theorem 1.1]{levin}, one has the following equalities on Poincar\'{e} series
\[\mathrm{P}^R_k(t)=\mathrm{P}^R_{S_2}(t)\cdot \mathrm{P}^{S_2}_k(t) \ \text{ and } \ \mathrm{P}^{S_1}_k(t)=\mathrm{P}^{S_1}_S(t)\cdot \mathrm{P}^S_k(t)\,,\] from which the desired equality is easily deduced. 
\end{remark}

\begin{chunk}\label{rem:hopf}
We denote by $\hopf$ the category of nonnegatively graded cocomutative Hopf $k$-algebras with codivided powers. That is, those Hopf algebras $A$ whose dual $\Hom_k(A,k)$ is a nonnegatively graded commutative Hopf algebra with divided powers; see \cite{Andre} for more details. This category becomes relevant to the present discussion as the universal enveloping algebra functor takes values in $\hopf$ and  restricts to an equivalence of categories when specializing to positively graded finite type (adjusted) Lie algebras over $k$; cf.\@ \cite{Andre,MM,Sjodin}.
\end{chunk}

The next result is now an immediate consequence of \cref{introthm1} and \cref{rem:hopf}. 

\begin{corollary}\label{cor:extalgebras}
Let    $S_1\leftarrow R\to S_2$ be a pair of Tor-independent surjective local ring maps, and set $S=S_1\otimes_R S_2.$ If at least one of the conditions in \cref{introthm1} holds, then $\Ext_{S}(k,k)$ is the pullback in $\hopf$ along the naturally induced maps $\Ext_{S_1}(k,k)\to \Ext_R(k,k)$ for $i=1,2.$
\end{corollary}

\begin{corollary}\label{cor:TorAlg}
With the notation and assumptions in \cref{cor:extalgebras}, there is an isomorphism of graded commutative Hopf algebras with divided powers
\[
\mathrm{Tor}^S(k,k)\cong\mathrm{Tor}^{S_1}(k,k)\otimes_{\mathrm{Tor}^R(k,k)}\mathrm{Tor}^{S_2}(k,k)\,.
\]
\end{corollary}
\begin{proof}
Consider the functor $\mathrm{Hom}_k(-,k)$ from $\hopf$ to the category of graded commutative Hopf algebras with divided powers. This functor takes pullbacks to pushouts.  
Moreover Tor algebras and Ext algebras are Hopf $k$-dual of each other, see \cite[Theorem 2.3.3]{GL}.
Therefore the corollary follows from \cref{cor:extalgebras}.
\end{proof}

\Cref{cor:TorAlg} generalizes  results of Avramov in \cite[Theorem~2]{tensor} and \cite[Corollary~5.6(a)]{small}.

Before stating the next Corollary, we recall the definition of depth for modules over a (not necessarily commutative) algebra.
Let $\mathcal{A}$ be a nonnegatively graded connected $k$-algebra and $\mathcal{M}$ a graded left $\mathcal{A}$-module. The \emph{depth} of $\mathcal{M}$ over $\mathcal{A}$ is
\[
\mathrm{depth}_{\mathcal{A}}\;\mathcal{M}=\mathrm{inf}\{n\geqslant0\mid\Ext_\mathcal{A}^n(k,\mathcal{M})\neq0\}.
\]

\begin{corollary}\label{cor:stable}
Let $S_1\leftarrow R\to S_2$ be a pair of Tor-independent minimal Cohen presentations, and set $S=S_1\otimes_R S_2$. If $S_1,S_2$ are singular local rings, then \[
\mathrm{depth}\;\Ext_S(k,k)=\mathrm{depth}\;\Ext_{S_1}(k,k)+\mathrm{depth}\;\Ext_{S_2}(k,k)\geqslant2.
\]
\end{corollary}
\begin{proof} For a singular local ring $A$, the depth of $\Ext_A(k,k)$ is at least one; see, for example, \cite[Lemma 5.1.7]{stable}. 
Now the desired result follows from the equalities below
\begin{align*}
\mathrm{depth}\;\Ext_S(k,k)&=\mathrm{depth}\;\mathrm{U}\pi^{\geqslant 2}(S)\\
&=\mathrm{depth}\;\pi^{\geqslant 2}(S)\\
&=\mathrm{depth}\;\pi^{\geqslant 2}(S_1)+\mathrm{depth}\;\pi^{\geqslant 2}(S_2)\\
&=\mathrm{depth}\;\Ext_{S_1}(k,k)+\mathrm{depth}\;\Ext_{S_2}(k,k).
\end{align*}
The first equality comes from \cite[Lemma 5.1.6]{stable} where $\mathrm{U}\pi^{\geqslant 2}(R)$ denotes the universal enveloping algebra of $\pi^{\geqslant 2}(R)$;  the third equality comes from equality \cref{e:isofiber} in the proof of \cref{thm:surjPiIsochar0}---combined with \cref{c:WeakCatGMin}---and \cite[Proposition 36.2]{rational}. The final equality follows again from \cite[Lemma 5.1.6]{stable}.
Since $S_1$ and $S_2$ are both singular, it follows that
\[
\mathrm{depth}\;\Ext_S(k,k)\geqslant 2.\qedhere
\]
\end{proof}

\begin{remark}\label{rem:stable}
A consequence of \Cref{cor:stable} and \cite[Theorem 7.2(3)]{stable}, which asserts if the depth of the Ext algebra of a local ring is at least two, is that the stable cohomology algebra of that ring has a particularly simple structure. Namely, the stable cohomology of a local ring is a direct sum of its Ext algebra and a specific torsion submodule; see the reference above for more details. Furthermore, if the ring $S$ is also assumed to be Gorenstein then its stable cohomology algebra is a trivial extension of its Ext algebra and a shift of the $k$-dual of its Ext algebra; cf.\@ \cite[Theorem 5.2]{ferraro}.
\end{remark}

We end this paper with one final application generalizing results of Quillen in \cite[Corollary~4.9 \& Remark~11.13]{Quillen2}, which can also be interpreted as a generalized ``dual" statement to a result of Quillen \cite[Corollary~7.5]{Quillen} (as well as \cite[Theorem~11.10]{Quillen}) involving Andr\'{e}--Quillen cohomology. The  latter result can also be deduced by an analogous result from local algebra  \cite[Theorem~A]{Moore}.

For a surjective local map $\vp\colon R\to S$ with residue field $k$, we write $\AQC  SRk$ for the Andr\'{e}--Quillen cohomology module of $\vp$ with coefficients in $k$; further details can be found in \cite{Iyengar:2007,Majadas/Rodicio:2010}. \cref{thm:surjPiIsochar0} can be recast in the theorem below by applying \cite[Section~6]{AH:1987}, a result essentially due to Quillen \cite[Theorem~9.5]{Quillen}, provided $k$ has charactersitic zero. 

\begin{corollary}\label{cor:AQ}
Let  $Q\to R\to S_i$ be surjective local ring maps  with residue field $k$ for $i=1,2$. If $S_1,S_2$ are Tor-independent over $R$ and $k$ has characteristic zero, then there is an isomorphism of graded Lie algebras 
\[
 \AQC  {S_1\otimes_R S_2}Qk\cong \AQC  {S_1}Qk \times_{\AQC  RQk} \AQC  {S_2}Qk\,.
\]
\end{corollary}

\bibliographystyle{amsplain}
\bibliography{biblio}

\end{document}